\documentclass{amsart}

\usepackage{amsmath}
\usepackage{amsfonts}
\usepackage{amssymb}
\usepackage{amsthm}

\theoremstyle{plain}
\newtheorem{theorem}{Theorem}
\newtheorem{lemma}{Lemma}

\def\R{{\mathbb R}}

\title{Connected Projections of Configuration Spaces of Linkages}
\author{Henry C.~King}

\begin{document}

\maketitle

\begin{abstract}
	We show that any compact connected semialgebraic set is the projection
	of a connected component of the configuration space of a linkage.

\end{abstract}

\section{Connected images of linkage configurations}

Recently Michael Kovalev asked me if any compact connected semialgebraic
set is the projection of a connected component of the configuration space
of a linkage, i.e., if we can choose some points on the linkage
which trace out all of $K$ and the linkage can go continuously from any point
of $K$ to any other.

This short note answers this in the affirmative, 
and also ties up a loose end in \cite{K1} for which I did not provide a reference 
(although a reference of which I am unaware undoubtedly exists).
See \cite{K1} or \cite{K2} for definitions of terms.

\begin{theorem}
Suppose $K\subset \R^m$ is a compact connected semialgebraic set.
Then for any $n\ge2$ there is a classical linkage in $\R^n$
with compact configuration space $C\subset \R^{n\ell}$ and a choice of $m$
coordinates of $\R^{n\ell}$ so that if 
$\pi\colon \R^{n\ell}\to \R^m$ is projection to these coordinates then
for any connected component $C_i$ of $C$, $\pi(C_i)=K$.
\end{theorem}

\begin{proof}
We know from Lemma 3.1 of \cite{K1} that there is a compact algebraic set 
$X\subset \R^m\times \R^k$ whose projection to $\R^m$ is exactly $K$, but $X$ may have
several connected components, each of which might only project to part of $K$.  
Choose such an $X$ with the least number of connected
components. If $X$ is not connected, there are two connected components $X_1$ and $X_2$
of $X$ whose projections to $\R^m$ intersect, so let $(x,y_i)\in X_i$ be points
with the same projection $x\in K$.
For convenience we may as well assume after an affine transformation that
$y_1=0$ and $y_2=(1,0,\ldots,0)$.
Consider the algebraic set $Y\subset \R^m\times \R^k\times \R$
given by $$Y=\{(u,v,t)\mid u=x, v_1^2+t^2=v_1, v_i=0\ i>1\}.$$
Then $Y$ is a circle which intersects both $X_1\times 0$ and $X_2\times 0$
and projects to $x\in K$.
Hence $Y\cup X\times 0$ is a compact algebraic set whose projection to $\R^m$
is $K$, but it has fewer connected components than $X$, a contradiction.
So $X$ was connected.
Now by Theorem 1.1 of \cite{K2} there is a linkage whose configuration space $C$
projects as a finite trivial analytic cover of $X$.
So in particular any connected component of $C$ projects isomorphically to $X$
and thus projects further onto $K$ as desired.
\end{proof}

I would imagine that the linkage in the proof above would move around in $X_1$,
then seem to mostly stop as it traverses $Y$, changing some internal gears so to speak,
and then continue moving in $X_2$.

\section{Unsemiing a semialgebraic set}

We'll say a subset $Y\subset \R^m$ has property P if there is a real algebraic set 
$X\subset \R^m\times \R^k$ whose projection to $\R^m$ is exactly $Y$.
By the Tarski-Seidenberg theorem we know a subset with property P is semialgebraic.
In Lemma 3.1 of \cite{K1} we asserted the converse and hinted at a proof,
but failed to give a reference which undoubtedly exists.
I'll prove this here, it's easier than finding a reference.
The proof follows from the following two Lemmas.

\begin{lemma}
	Suppose $Y,Z\subset \R^m$ have property P.
	Then $Y\cup Z$, $Y\cap Z$, and $Y-Z$ each have property P.
\end{lemma}

\begin{proof}
	Let $X\subset \R^m\times \R^k$ and $W\subset \R^m\times\R^\ell$ be real algebraic sets
	which project to $Y$ and $Z$ respectively.
	Let $p:\R^m\times \R^k\to\R$ and $q\colon \R^m\times \R^\ell\to \R$ be polynomials
	so that $X=p^{-1}(0)$ and $W=q^{-1}(0)$.
	Then the algebraic set $$X\times 0\cup \{(x,0,z)\in \R^m\times\R^k\times\R^\ell \mid  q(x,z)=0  \}$$
	projects to $Y\cup Z$,
	the algebraic set $$\{(x,y,z)\in \R^m\times\R^k\times\R^\ell \mid  q(x,z)=0, p(x,y)=0  \}$$
	projects to $Y\cap Z$, and the algebraic set
	$$\{(x,y,z,t)\in \R^m\times\R^k\times\R^\ell\times\R \mid  tq(x,z)=1, p(x,y)=0  \}$$
	projects to $Y-Z$.
\end{proof}

\begin{lemma}
	Let $p\colon\R^m\to \R$ be a polynomial.
	Then $p^{-1}(0)$ and $p^{-1}([0,\infty))$ each have property P.
\end{lemma}

\begin{proof}
	Note $p^{-1}(0)$ has property P since it is the projection of itself.
	Also $p^{-1}([0,\infty))$ is the projection of the algebraic set
	$\{(x,t)\in \R^m\times\R \mid t^2 = p(x)\}$.
\end{proof}

\end{document}